\documentclass[11pt]{amsart}
\usepackage{pstricks,pst-plot,}
\usepackage{color}
\usepackage{ulem}
\usepackage{cancel}

\definecolor{Black}{cmyk}{0,0,0,1}
\definecolor{OrangeRed}{cmyk}{0,0.6,1,0} 
\definecolor{DarkBlue}{cmyk}{1,1,0,0.20}
\definecolor{myblue}{rgb}{0.66,0.78,1.00}
\definecolor{Violet}{cmyk}{0.79,0.88,0,0}
\definecolor{Lavender}{cmyk}{0,0.48,0,0}

\parskip=\smallskipamount

\newtheorem{theorem}{Theorem}[section]
\newtheorem{lemma}[theorem]{Lemma}
\newtheorem{corollary}[theorem]{Corollary}
\newtheorem{proposition}[theorem]{Proposition}

\theoremstyle{definition}

\newtheorem{remark}[theorem]{Remark}

\newcommand{\C}{\mathbb{C}}

\newcommand{\bea}{\begin{eqnarray*}}
\newcommand{\eea}{\end{eqnarray*}}

\numberwithin{equation}{section}

\begin{document}
\title[Weighted approximation in $\mathbb{C}$ ]{Weighted approximation in $\mathbb{C}$ }
\author{Jujie Wu}
\author{John Erik Forn\ae ss}

\address{ {
$^{*}$ Corresponding author  \\
Jujie Wu} \
\\{\it E-mail address:}  jujie.wu@ntnu.no
\\{ School of Mathematics and statistics, Henan University}\\{ Jinming Campus of Henan University, Jinming District, City of Kaifeng, Henan Province. P. R. China, 475001}   \\ \ \ \ \  \ \ \ \ \ \ \ \ \ \ \ \ \ \ \ \ \ \ \ \ \ \ \ \ \ \ \  \ \ \ \ \ \ \  \ \ \ \ \ \ \  \ \ \ \ \ \ \   \ \ \    \ \ \ \ \ \ \ \ \ \  \ \ \ \ \ \ \ \ \ \  \ \ \ \ \ \ \ \ \ \ \ \ \ \ \ \ \ \ \ \ \ \ \ \ \ \ \ \ \ \ \ \  \ \ \ \ \ \ \ \ \ \ \ \ \ \ \ \ \ \ \  \ \ \ \ \ \ \ \ \ \ \ \ \ \ \ \ \ \ \ \ \ \ \ \ \ \ \ \ \ 
 \  \ \ \  \ \ \ \ \ \ \ \  \ \ \ \ \ \ \ \ \ \ \ \ \    \\{Department of Mathematical Sciences, NTNU}\\{ Sentralbygg 2, Alfred Getz vei 1, 7034 Trondheim, Norway} }


\address{{John Erik Forn\ae ss}
\\{\it E-mail address:} john.fornass@ntnu.no
\\{ Department of Mathematical Sciences, NTNU}\\{ Sentralbygg 2, Alfred Getz vei 1, 7034 Trondheim, Norway}}

\date{}
\maketitle

\bigskip

\begin{abstract}
We prove that if $\{ \varphi_j\}_j$ is a sequence of subharmonic functions which are increasing to some subharmonic function $\varphi$ in $\mathbb{C}$, then the union of all the weighted Hilbert spaces $H(\varphi_j)$ is  dense in the weighted Hilbert space $H(\varphi)$. 

\bigskip


\smallskip

\end{abstract}
\tableofcontents

\section{Introduction}

Suppose $\varphi $ is a measurable function, locally bounded above on a domain $\Omega \subset \mathbb{C}^n$. 
 Set 
$$
H(\Omega, \varphi) : = \{ f \in\mathcal{O}(\Omega) : 
\int_{\Omega}|f|^2e^{-\varphi}d\lambda < + \infty  \}
$$
where $\mathcal{O}(\Omega)$ stands for the space of holomorphic functions on $\Omega$ and $d\lambda$ is the Lebesgue measure. 
Especially let $H(\varphi)$ denote the space of entire functions $f$ with $L^2(\mathbb{C}^n, \varphi)$ norm, i.e., $\|f\|^2_{L^2(\mathbb{C}^n, \varphi)}=\int_{\mathbb {C}^n}|f|^2e^{-\varphi}d\lambda < + \infty.$
 
 Since $f\in H(\Omega,\varphi)$,  the function $|f|^2$ is plurisubharmonic (psh) and $\varphi$ is locally bounded above, we have
\begin{eqnarray} \label{ineq:subpro}
| f(w)| \le \frac {C_n}{r^n} \|f\|_{L^2(B(w,r),0)} \le  \frac {C'_n}{r^n} \|f\|_{H(\Omega,\varphi)},\end{eqnarray}
if the ball  $B(w,r) \subset \subset \Omega$ and for $K \subset \subset \Omega$ 
$$
\sup\limits_{K} |f| \le C \|f\|_{H(\Omega,\varphi)},
$$
where $C$ depends only on $K$ and $\Omega$. So $H(\Omega,\varphi)$ is a closed subspace of $L^2(\Omega, \varphi)$ and thus a Hilbert space. Let $ K_{\Omega, \varphi} (z, w)$ denote the weighted Bergman kernel corresponding to the Hilbert space $H(\Omega,\varphi)$. If $\varphi = 0$, then $K_\Omega (z, w): = K_{\Omega, 0}(z, w) $ is the classical kernel introduced by Stefan Bergman.

 In 1971, B. A. Taylor \cite{Taylor1971} investigated weighted approximation results for entire functions on $\mathbb C^n$. He proved:

\begin{theorem}
Let $\varphi_1\le \varphi_2\le \varphi_3 \le \cdots$ be psh functions on $\mathbb C^n,$ assume
$\varphi=\lim\limits_{j\rightarrow \infty} \varphi_j$ is psh, and suppose that
$\int_Ke^{-\varphi_1}d\lambda<\infty$ for every compact set $K.$ Then the closure of $\bigcup \limits _{j=1}^\infty H(\varphi_j+\log(1+\|z\|^2))$  in the Hilbert space $L^2( \varphi + \log(1+\|z\|^2))$ contains $H(\varphi)$.
\end{theorem}

In \cite{Fornaess2017} we improved Taylor's result as follows 
\begin{theorem}  \label{th:main2}
Let $\varphi_1\leq \varphi_2\leq \varphi_3\leq \cdots $ be psh functions on $\mathbb{C}^n$ with psh limit. For any $\epsilon > 0,$ let $\widetilde{\varphi}_j = \varphi_j +  \epsilon \log(1+\|z\|^2)$ and $\widetilde{\varphi} = \lim\limits_{j\rightarrow+\infty} \widetilde{\varphi }_j$. Then $\bigcup \limits _{j=1}^\infty H(\widetilde{\varphi}_j)$ is dense in $H(\widetilde{\varphi})$.
\end{theorem}
It is an important question whether this theorem is true or false when $\epsilon = 0$. Here by using some potential theoretic properties of subharmonic functions we show that it holds in one dimension. 

\begin{theorem} \label{th:main1}
Let $\varphi_1\leq \varphi_2\leq \varphi_3\leq \cdots $ be subharmonic functions on $\mathbb{C}$. Suppose that $\varphi = \lim\limits_{k} \varphi_k$ and $\varphi$ is locally bounded above. Then $\bigcup \limits _{k=1}^\infty H(\varphi_k)$ is dense in $H(\varphi)$.
\end{theorem}


\begin{remark}
Let $\varphi = \lim \limits_k \varphi_k$. We define $\phi = \varphi ^{\ast} : = \limsup\limits _{\zeta \rightarrow z } \varphi(\zeta)$, $ z \in \mathbb{C}$, which is the upper regularization of $\varphi$. Then $\phi$ is subharmonic, $\varphi\leq \phi$, $\varphi = \phi $ almost everywhere on $\mathbb{C}$ and we have 
$H(\varphi) = H(\phi)$ (See Theorem 3.4.2 in \cite{Rans95}).
\end{remark}

We need the strong openness theorem as follows, see \cite{guanzhou1401},  \cite{guanzhou1311}, \cite{guanzhou1501}, \cite{Lempert2017}.


\begin{theorem}[Strong openness theorem]\label{th:opennessconjecture}
Let $V \subset \subset U \subset \mathbb{C}^n$ be an open set. Let $\varphi_1\leq \varphi_2\leq \varphi_3\leq \cdots $ be non-positive psh functions on $U$ such that $\varphi = \lim\limits_{k} \varphi_k$ and $\varphi$ is locally bounded above. If $f \in \mathcal{O} (U)$ such that 
$$\int _U |f|^2 e^{-\varphi} d \lambda < \infty,$$
then there exists $j_0$ so that when $j \geq j_0$ 
$$
\int_ V |f|^2 e^{ - \varphi_j} d\lambda < \infty.
$$  
\end{theorem}

For convenience, we will use $K_{\varphi_j}(z, w)$(resp. $K_\varphi(z,w )$) to denote the weighted Bergman kernel corresponding to the Hilbert space $H(\varphi_j):= H(\mathbb{C}, \varphi_j)$ (resp. $H( \varphi): = H(\mathbb{C},  \varphi)$). As an application of the approximation theorem \ref{th:main1}, we will prove 

\begin{theorem} \label{th:bergmanconver}
Let $\varphi_1\leq \varphi_2\leq \varphi_3\leq \cdots $ be subharmonic functions on $\mathbb{C}$. Suppose that $\varphi = \lim\limits_{k} \varphi_k$ and $\varphi$ is locally bounded above. Then

$$
 \lim\limits_{j}K_{\varphi_j}(z,z ) =K_\varphi(z,z), \ \ \   \forall  z \in \mathbb{C}. 
$$
\end{theorem}
Ligocka showed that the classical Bergman kernel of certain Hartogs domains can be expressed as the sum of a series of weighted Bergman kernels defined on another domain of lower dimension. We set
$$
\Omega_j =\{(z,w)\in \mathbb{C} \times  \mathbb{C}: |w| < e^{-\varphi_j(z)} \}
$$
and
$$
\Omega =\{(z,w)\in \mathbb{C} \times  \mathbb{C}: |w| < e^{-\phi(z)} \}.
$$

 We note that a domain $\{ (z,w)\in \mathbb{C} \times  \mathbb{C}:  |w|<e^{-\ell (z)}\}$ is open exactly when $\ell(z)$ is upper semicontinuous. Hence we use $\phi$ in the definition of $\Omega$ and not $\varphi.$

We denote by $K_{\Omega_j} [(z,t), (w,s)]$ (resp. $K_{\Omega}[(z,t), (w,s)]$)  the Classical Bergman kernel of the Hilbert space $H(\Omega_j, 0)$(resp. $H(\Omega, 0)$), where $z ,t , w, s \in \mathbb{C}$.
Ligocka's formula \cite{Ligocka1989} implies that 
$$
K_{\Omega_j}[(z,t), (w,s)] = \sum\limits_{k=0}^{\infty} 2(k+1) K_{2(k+1)\varphi_j}(z,w) \langle t, s\rangle^k, \ \  z ,t , w, s \in \mathbb{C}
$$
where $ K_{2(k+1)\varphi_j}(z,w)$ is the weighted Bergman kernel of the Hilbert space $H(\mathbb{C}, 2(k+1) \varphi_j)$.


According to the result of Theorem \ref{th:bergmanconver} we obtain that 

\begin{theorem} \label{Hartogsconvergence}
Let $\{ \varphi_j \}$ and $\varphi$ be as in theorem \ref{th:bergmanconver}, then the sequence 
$K_{\Omega_j}[(z,t), (w,s)]$  converges to $ K_{\Omega}[(z,t), (w,s)]$ locally uniformly in $\Omega \times \Omega$.
\end{theorem}

The set-up of the paper is as follows:  We prove in Section 2 an integral estimate for subharmonic weights. In Section 3 we discuss some potential theoretic properties of subharmonic functions which we will need. In Section 4, we prove the Main Theorem, Theorem \ref{th:main1}. In Section 5, we prove the convergence of the Bergman Kernel which can be seen as an application of the main Theorem \ref{th:main1}.

\section{Subharmonic functions in $\mathbb C$}



\begin{lemma}
Let $\alpha_i>0, z_i\in \mathbb C, i=1,\dots,N$ and let $\alpha=\sum_i \alpha_i.$
For any $z$ which is not one of the $z_i$ we have the inequality
$$
\Pi_{i=1}^N \left(\frac{1}{|z-z_i|} \right)^{\alpha_i}
\leq \sum_{i=1}^N \frac{\alpha_i}{\alpha} \left(\frac{1}{|z-z_i|} \right)^{\alpha}.
$$
\end{lemma}

\begin{proof}
\bea
\Pi_{i=1}^N \left(\frac{1}{|z-z_i|} \right)^{\alpha_i} 
& = & e^{\sum_{i=1}^N \frac{\alpha_i}{\alpha}\log \left( \left(\frac{1}{|z-z_i|} \right)^{\alpha}   \right)}\\
& & {\mbox{Since exp is convex, }}\\
& \leq & \sum_{i=1}^N \frac{\alpha_i}{\alpha}e^{\log \left( \left(\frac{1}{|z-z_i|} \right)^{\alpha}   \right)}\\
& = & \sum_{i=1}^N \frac{\alpha_i}{\alpha} \left(\frac{1}{|z-z_i|} \right)^{\alpha} .  \\
\eea
\end{proof}

\begin{lemma}
Let $|z_0|<R$  and suppose $0<\alpha<2.$
Then
$$
\int_{|z|<R}\left(\frac{1}{|z-z_0|} \right)^{\alpha} d\lambda(z)\leq \frac{28R^2}{2-\alpha}.
$$
\end{lemma}
\begin{proof} 
\bea
 \int_{|z|<R}\left(\frac{1}{|z-z_0|} \right)^{\alpha} d\lambda(z) 
& \leq & \int_{|z|<2R}\left(\frac{1}{|z|} \right)^{\alpha} d\lambda(z) \\
& = & 2\pi\frac{(2R)^{2-\alpha}}{2-\alpha}\\
& \leq & 28R^2/(2-\alpha).
\eea
\end{proof}
As a direct consequence, we obtain
\begin{theorem} \label{th: theorem2.1}
Let $\alpha_i>0, |z_i|<R, i=1,\dots,N$ and let $\alpha=\sum_i \alpha_i.$
Suppose that $\alpha<2.$ Then
$$
\int_{|z|<R}\Pi_{i=1}^N \left(\frac{1}{|z-z_i|} \right)^{\alpha_i}d\lambda(z)
\leq \frac{28R^2}{2-\alpha}.
$$
\end{theorem}


\begin{lemma} \label{le:measuretheory}
Let $K$ be a compact set in $\mathbb R^n.$ Suppose $f(x,y):K\times K\rightarrow \mathbb R$ is continuous. Let $\mu$ be a positive measure on $K$ with finite total mass $\alpha.$
Let $\epsilon>0.$ Then there exists a finite positive measure $\sigma=\sum_{i=1}^N r_i \delta_{p_i}$
with $\sigma(K)=\alpha$ where the $p_i$ are points in $K.$ Moreover 
if $\phi(x)=\int_{y\in K} f(x,y)d\mu(y)$ and $\psi(x)=\int_{y\in K} f(x,y) d\sigma(y)$ then
$\psi(x)<\phi(x)+\epsilon.$
\end{lemma}
\begin{proof}
Divide $K$ into finitely many small sets $K_i$ and pick $p_i\in K_i$. By uniform continuity
of $f$ we can assume that for any $y\in K_i$ we have that
$|f(x,y)-f(x,p_i)|<\epsilon/C$ for any given $C$ (which will depend on $\alpha.$)
Let $r_i=\mu(K_i).$ Define $\sigma=\sum_{i=1}^N r_i \delta_{p_i}$. Let $x\in K.$ Define
\bea
\psi(x) & = & \int_{y\in K} f(x,y)d\sigma(y) \\
& = & \sum_i f(x,p_i)r_i\\
& = & \sum_i f(x,p_i) \int_{K_i}d\mu(y)\\
& = & \sum_i \int_{K_i} f(x,p_i)d\mu(y)\\
& \leq & \sum_{i} \int_{K_i} (f(x,y)+\epsilon/C)d\mu\\
& = & \phi(x)+(\epsilon/C) \int_K d\mu\\
& = & \phi(x)+\epsilon \alpha/C.\\
\eea
We let $C=\alpha$. 
\end{proof}

\begin{theorem} \label{ co:estimateforintegral}
Let $d\mu$ be a positive measure on the disc of radius $R$ with total mass $\alpha<2.$
Set $\phi(z)=\int_{|\zeta|<R} \log |z-\zeta| d\mu(\zeta).$
Then
$$
\int_{|z|<R} e^{-\phi(z)}d\lambda(z) \leq \frac{28R^2}{2-\alpha}.
$$
\end{theorem}

\begin{proof}
Define $\psi_n(z,\zeta)=\max\{\log|z-\zeta|,-n\}.$ Define 
$$
\phi_n(z)=\int_{|\zeta|<R}\psi_n(z,\zeta)d\mu(\zeta) \ \ \   \text{for} \ \ \  z\in \Delta(R).
$$ 
Then $\phi_n:\Delta(R) \rightarrow \mathbb R$ is continuous and $\phi_n\searrow \phi$ pointwise
for $z\in \Delta(R).$ Hence $e^{-\phi_n(z)}\nearrow e^{-\phi}$ on $\Delta(R).$
Therefore, in order to show that 
$$
\int_{|z|<R} e^{-\phi(z)}d\lambda(z)\leq \frac{28R^2}{2-\alpha},
$$ 
it suffices to show
that
$$
\int_{|z| < R} e^{-\phi_n(z)}d\lambda(z) \leq \frac{28R^2}{2-\alpha}+\frac{1}{n}\;\ \ \ \ \   \forall n.
$$
We fix $n$. Let $\delta>0.$ Since $\psi_n$ is continuous, according to Lemma \ref{le:measuretheory} we can find a finite positive measure
$\mu_n=\sum\limits_{i =1}^{N}\alpha_i \delta_{z_i}$ with total mass $\alpha$ so that
$$
\tilde{\phi}_n:= \int_{|\zeta|<R}\psi_n(z,\zeta)d\mu_n(\zeta)\leq \phi_n(z)+\delta.$$
By Theorem \ref{th: theorem2.1} we know that
$$
\int_{|z|<R}\Pi_{i=1}^N \left(\frac{1}{|z-z_i|} \right)^{\alpha_i}d\lambda(z)
\leq \frac{28R^2}{2-\alpha}.
$$
Hence
$$
\int_{|z|<R} e^{-\sum_i \alpha_i \log|z-z_i|}d\lambda \leq \frac{28R^2}{2-\alpha}.
$$
So
$$
\int_{|z|<R} e^{-\int_{|\zeta|<R} \log |z-\zeta|d\mu_n(\zeta)}d\lambda(z) \leq \frac{28R^2}{2-\alpha}.
$$
Notice that $\max\{\log|z-\zeta|,-n\}\geq \log|z-\zeta|$, then we have
$$
-\max\{\log|z-\zeta|,-n\}\leq -\log|z-\zeta|.
$$
Hence
$\int_{|z|<R}e^{-\tilde{\phi}_n}d\lambda \leq \frac{28R^2}{2-\alpha}.$
Choosing $\delta$ small enough we get that
$$
\int_{|z|<R} e^{-\phi_n}d\lambda \leq \frac{28R^2}{2-\alpha}+\frac{1}{n}.
$$

\end{proof}

\section{Comparison of weights}

\begin{lemma}\label{le:guji}
Suppose that $|\zeta|<R$ and $z\in \mathbb C.$
Then
$$
\log|z-\zeta| \leq \frac{1}{2}\log (1+|z|^2)+ \log 2+\frac{1}{2}\log(1+R^2).
$$
\end{lemma}

\begin{proof}
\bea
\log|z-\zeta| & \leq & \log(|z|+|\zeta|)\\
&\leq & \log 2+\max \{\log|z|,\log|\zeta|\}\\
& \leq & \log 2+\frac{1}{2}\log(1+|z|^2)+\frac{1}{2}\log(1+|\zeta|^2).\\
\eea
\end{proof}

\begin{proposition} \label{prop phipsi}
Let $\mu$ be a nonnegative measure on the disc $|\zeta|<R$ with mass $M.$
Let $\phi_1(z)=\int_{|\zeta|<R}\log|z-\zeta| d\mu(\zeta)$ and $\phi_2(z)=M/2 \log (1+|z|^2).$
Suppose that $\phi=\phi_1+\sigma$ and $\psi=\phi_2+\sigma.$ Then there exists constant
$C$ so that $\|f\|_\psi^2\leq C\|f\|_\phi^2$. In particular, $H(\phi)\subset H(\psi).$
\end{proposition}

\begin{proof}
By the previous lemma we know $$\phi_1\leq \phi_2+M\left(\log 2+\frac{1}{2}\log(1+R^2)\right),$$
hence $$\phi\leq \psi+M\left(\log 2+\frac{1}{2}\log(1+R^2)\right).$$
It follows that $e^{-\psi} \leq C e^{-\phi}$. Thus we have $H(\phi)\subset H(\psi).$
\end{proof}

For the other direction we need some extra hypothesis.

\begin{proposition} \label{prop psi phi}
Let $\mu$ be a nonnegative measure on the disc $|\zeta|<(R+\epsilon)$ ($ \epsilon >0 $ some constant) with mass $\beta\in (0,2).$ Let $\alpha$ be the $ \mu$ mass of $\Delta(R).$
Suppose that $\phi$ is subharmonic on $\mathbb C$ and that $\frac{1}{2\pi}\Delta \phi=\mu$ on  $\Delta(R+\epsilon).$
Let $\phi_1(z)=\int_{|\zeta|<R}\log|z-\zeta| d\mu(\zeta)$ and $\phi_2(z)=\alpha/2 \log (1+|z|^2).$ Write  $\phi=\phi_1+\sigma$ and $\psi=\phi_2+\sigma.$
Then $H(\psi)\subset H(\phi).$
\end{proposition}


We prove first a lemma:

\begin{lemma} \label{le: estimateforlog}
There exists a constant $C$ so that if $|z|\geq R+\frac{\epsilon}{2}$ and $|\zeta|<R,$
then $$
\log|z-\zeta|\geq \frac{1}{2}\log(1+|z|^2)-C.$$
\end{lemma}

\begin{proof}
We get
\bea
\log|z-\zeta| & = & \log |z|\left|1-\frac{\zeta}{z}\right|\\
& \geq & \log|z|+\log\left(1-\frac{2R}{2R+\epsilon}\right)\\
&\geq & \frac{1}{2}\log \left(|z|^2\left(\frac{1}{R^2}+1\right)\right)-\frac{1}{2}\log \left(\frac{1}{R^2}+1\right)\\
&& + \log\left(1-\frac{2R}{2R+\epsilon}\right)\\
& \geq & \frac{1}{2}\log (1+|z|^2)-\frac{1}{2}\log \left(\frac{1}{R^2}+1\right)+\log\left(1-\frac{2R}{2R+\epsilon}\right).
\\
\eea
The proof gives that we can choose $C=\frac{1}{2}\log \left(\frac{1}{R^2}+1\right)-\log(1-\frac{2R}{2R+\epsilon})$.
\end{proof}
In order to prove the above Proposition \ref{prop psi phi} we also need the following well known Riesz Decomposition Theorem (see Ransford \cite{Rans95} Theorem 3.7.9).

\begin{theorem} [ Riesz Decomposition Theorem]
Let $u$ be a subharmonic function on a domain $D$ in $\C$, with $u \not\equiv -\infty$. Then,
given a relatively compact open subset $U$ of $D$, we can decompose $u$ as 
$$
u= \int_{ \zeta \in U } \log |z- \zeta | d \mu(\zeta) +  h  
$$
on $U$, Where $\mu = \frac{1}{2 \pi} \Delta u |_U$ and $h$ is harmonic on $U$.
\end{theorem}

We prove Proposition \ref{prop psi phi}:
\begin{proof}
Let $F$ be an entire function so that $\int_{\C} |F|^2 e^{-\psi} d\lambda<\infty.$   If $|z| > R+ \frac \epsilon 2$, by the previous Lemma, 
\bea
\phi_1& = & \int_{|\zeta|<R}\log|z-\zeta| d\mu(\zeta) \\
& \geq &\frac{\alpha}{2}\log (1+|z|^2) -\alpha C \\
&= & \phi_2 -\alpha C. 
\eea
Here $C$ is the explicit constant from Lemma \ref{le: estimateforlog}.
It follows
that 
$$
\int_{|z|\geq R+\epsilon/2}|F|^2e^{-\phi} d\lambda< e^{\alpha C}\int_{|z|\geq R+\epsilon/2 } |F|^2 e^{-\psi} d\lambda < +\infty.
$$
On the disc of radius $R+\epsilon$, according to Riesz decomposition theorem we can write $\phi(z)=\int_{|\zeta|<R+\epsilon} \log|z-\zeta| d\mu(\zeta)
+\tau:=\Phi+\tau$ where $\tau$ is a subharmonic function which is harmonic on the disc of radius $R+\epsilon$.
In particular we have that $\tau$ is uniformly bounded on the disc of radius $R+\frac23\epsilon.$
The result follows from Theorem \ref{ co:estimateforintegral} that 
$$
\int_{|z|<R+\frac12\epsilon}e^{-\Phi}d\lambda< \int_{|z|<R+\epsilon}e^{-\Phi}d\lambda<\frac {2\pi}{ 2-\beta} (2(R+ \epsilon))^2 < + \infty
$$
and hence the same is true for the integral $|F|^2e^{-\phi}$ on the disc of radius $R+\frac 12 \epsilon.$ 
That means we have $\int_\C |F|^2 e^{-\phi} d \lambda < +\infty$. Thus 
$$
H(\psi)\subset H(\phi).
$$
\end{proof}

\section{Proof of Theorem \ref{th:main1}}

In this section we prove the Main Theorem, Theorem \ref{th:main1}.
We prove first the case when the upper regularization $\phi$ is a harmonic function. We can suppose that $\varphi_1$ is not identically $-\infty.$

\begin{lemma}
If $\phi $ is harmonic, then there are constants $c_1\leq c_2 \leq \cdots, c_j\rightarrow 0$ so that
$\varphi_j=\varphi+c_j=\phi+c_j$.
\end{lemma}

To prove the lemma, observe that $\varphi_j-  \phi \leq \varphi_j-  \varphi \leq 0$. Then $\varphi_j-\phi$ must be constant. Thus we have $\varphi=\phi$. The Lemma follows.

Then the theorem follows in the case when $\phi$ is harmonic.

We can generalize this to the following case:

Condition (A): The upper regulatization $\phi$ of $\varphi$ is a subharmonic function with the following property:
$\Delta \phi= \sum_i a_i\delta_{z_i}$ where $z_i$ is a  sequence in $\mathbb C$ and $a_i> 0$.

\begin{lemma}
In the case of (A), there exist nonpositive constants $c_j$ so that $\varphi_j=\varphi+c_j=\phi+c_j$.
\end{lemma}

\begin{proof}
Fix $N.$ We can write
$$
\phi=\sum\limits_{i=1}^N a_i \log|z-z_i|+\psi_N
$$
where $\psi_N$ is subharmonic and $\Delta \psi_N = \sum\limits_{i>N}a_i \delta_{z_i}$. We get for any $j, N$ 
$$
\varphi_j-\sum\limits_{i=1}^N a_i\log|z-z_i|\leq \psi_N
$$
for $z\neq z_i$. Then $\varphi_j-\sum\limits_{i=1}^N a_i\log|z-z_i|$  extends across $z_i$
as a subharmonic function $\psi_{j}^N.$ That is $\varphi_j=\sum\limits_{i=1}^Na_i\log|z-z_i|+\psi_{j}^N$ for any $N$ on $\mathbb C.$
Thus $\varphi_j=-\infty$ at $z_i$ and $\Delta \varphi_j\geq \sum\limits_{i=1}^N a_i\delta_{z_i}.$ It follows that $\Delta \varphi_j
\geq \Delta \phi$ on $\mathbb C$ for all $j.$ So we can find some subharmonic function $\lambda_j$ such that
 $\varphi_j=\phi +\lambda_j$.
But $\lambda_j\leq 0$ thus it must be constant. The Lemma follows and hence the theorem also follows in this case.
\end{proof}

Condition (B):
Let $\phi$ be a subharmonic function on $\mathbb C.$ Let $\mu$ denote the Laplacian of $ \frac {1}{2\pi}\phi.$
We say that $\phi$ satisfies condition (B) if there exist some constant $R>0$ and  $c>0$ such that on the disc $|\zeta| < R+ c$,  the mass of $\mu$ is equal to $\beta$, with $0<\beta <2$ and the mass of $\mu$ on the disc $|\zeta|<R$ is $\alpha>0.$
According to Proposition \ref{prop phipsi} and Proposition \ref{prop psi phi} with the same notation as there  we have the following Corollary:

\begin{corollary} \label{cor:equanorm}
If $\phi$ satisfies the above condition (B), then the spaces $L^2(\mathbb{C}, \phi)$ and $L^2(\mathbb{C}, \psi)$ are the same. Moreover the norms are equivalent.
\end{corollary}


 \begin{lemma}\label{ConditionB}
 Let $\varphi_1\leq \varphi_2\leq \varphi_3\leq \cdots $ be subharmonic functions on $\mathbb{C}$ with $\varphi = \lim\limits_{k} \varphi_k$. Suppose the upper regulatization $\phi$ of $\varphi$ satisfies the above Condition (B).
Then $\bigcup \limits _{k=1}^\infty H(\varphi_k)$ is dense in $H(\varphi)$.
\end{lemma}

To prove Lemma \ref{ConditionB} we need the following $L^2$-estimate by Berndtsson (see \cite{Boberndtsson2001}).

\begin{lemma} [\cite{Boberndtsson2001}]\label{Boberndtsson2001}
Let $\Omega \subset \mathbb{C}^n$ be a pseudoconvex domain and $\varphi \in psh(\Omega)$.  Suppose $\psi$ is a $C^2$ real function satisfying
$$
r i\partial \overline{\partial} (\varphi+\psi) \geq i\partial \psi \wedge \overline{\partial} \psi
$$
in the sense of distributions for some $0<r<1$. Then for each $\overline{\partial}$-closed $(0,1)$-form $v$, there is a  solution $u$ of $\overline{\partial} u =v$ which satisfies
\begin{eqnarray}\label{ineq:estimate1}
\int _{\Omega} |u|^2e^{\psi-\varphi}d\lambda \leq \frac{6}{(1-r)^2} \int _{\Omega} |v|_{i\partial \overline{\partial}(\varphi+\psi)}^2e^{\psi-\varphi}d\lambda
\end{eqnarray}
in the sense of distributions.
\end{lemma}

\begin{proof} [ Proof of Lemma \ref{ConditionB} ] 
Put {\bf$\frac{1}{2\pi} \Delta \phi = \mu$}, $\frac{1}{2\pi} \Delta \varphi_j = \mu_j$ for each $j$. The mass of $\mu$ (resp. $\mu_j$) on the disc $\Delta(R)$ will be denoted by $\alpha$ (resp. $\alpha_j$).
Since $\varphi_j$ is increasing to $\varphi$ and $\varphi=\phi$ a.e., we have that $\Delta \varphi_j$ converges to $\Delta \phi$ in the sense of distributions. That means we can find some $0<c' < c$ with the mass of $\mu_j$ on the disc $\Delta(R+c') $  belongs to $(0,2)$ when $j$ is sufficient large. Moreover the mass of $\mu_j$ on the disc $\Delta(R)$ is $\alpha_j>\alpha/2$ for all sufficiently large $j.$
By using  Riesz decomposition theorem we can write
$$
 \phi = \widetilde{\phi}+  \int_{|\zeta| < R} \log |z-\zeta|d \mu(\zeta), \ \ \ 
\varphi_j = \widetilde{\varphi}_j +  \int_{|\zeta| < R} \log |z-\zeta|d \mu_j(\zeta), \ \ \  \forall j.
$$
Here $\widetilde{\phi}$ and $\widetilde{\varphi}_j$ are subharmonic functions on $\mathbb{C}$.
Put
$$
\varphi'_j = \widetilde{\varphi}_j + \frac{\alpha_j}{2} \log(1+|z|^2), \ \ \  \forall j
$$
and 
$$
\phi'= \widetilde{\phi} + \frac{\alpha}{2} \log(1+|z|^2).
$$
By Corollary \ref{cor:equanorm} we have that $H(\varphi'_j) = H(\varphi_j)$ for each $j$ and $H(\phi')=H(\phi) = H(\varphi).$
The following proof is very similar to \cite{Fornaess2017}, we make some changes.
Here for convenience we replace $\log(1+|z|^2)$ by $\log(e+|z|^2)$ without changing the spaces and the norms because of equivalence.
Let $\chi : \mathbb{R} \rightarrow [0,1]$ be a smooth function on $\mathbb{C}$ satisfying $\chi | _{(-\infty, \log \frac 12)} = 1$, $\chi | _{(0, +\infty)} = 0$ and $|\chi '| \leq 3$. Set
$$
\psi = - \log\left(\log(e+|z|^2)\right).
$$
Then we have
$$
i \partial\overline{\partial} \psi = - i\frac{\partial\overline{\partial} \log(e+|z|^2)}{\log(e+|z|^2)} + i\frac{ \partial \log(e+|z|^2) \wedge \overline{\partial}\log(e+|z|^2) }{  (\log(e+|z|^2))^2}
$$
and
$$
i \partial \psi\wedge  \overline{\partial}\psi=  i\frac{ \partial \log(e+|z|^2) \wedge \overline{\partial}\log(e+|z|^2)}{(\log(e+|z|^2))^2}.
$$
For each $j$, put $\Phi_j: = \varphi'_j +\frac{\alpha_j}{4} \psi = \widetilde{\varphi}_j  + \frac{\alpha_j}{2} \log(e+|z|^2)+ \frac{\alpha_j}{4} \psi$
and $\Psi_j =  \frac{\alpha_j}{4} \psi$. By calculation we have 
\begin{eqnarray}
 i \partial\overline{\partial}  (\Phi_j +\Psi_j ) 
&\geq & \frac{\alpha_j}{2} i\partial\overline{\partial} \log(e+|z|^2) -\frac{\alpha_j}{2}i \frac{  \partial\overline{\partial} \log(e+|z|^2)}{\log(e+|z|^2)} \nonumber \\
&&+ \frac{\alpha_j}{2}  i\frac{ \partial \log(e+|z|^2) \wedge \overline{\partial}\log(e+|z|^2) }{  (\log(e+|z|^2))^2} \nonumber \\
&\geq &\frac{\alpha_j}{2}  \partial \psi \wedge \overline{\partial} \psi \nonumber\\
& = &  \frac{8}{ \alpha_j }  \partial \Psi_j \wedge \overline{\partial} \Psi_j \nonumber
\end{eqnarray}
while $0<\frac{\alpha_j}{8}<1$.
Let $f\in H( \varphi ) = H( \phi')  $.  We fix an $\epsilon >0$.
Put
$$
v_\epsilon : = f \cdot \overline{\partial}\chi\left( \log (-\psi)+ \log \epsilon  \right).
$$
Apply Lemma \ref{Boberndtsson2001} especially for $\Omega = \mathbb{C}$ with $\varphi$ and $\psi$ replaced by $\Phi_j$ and $\Psi_j$ respectively, we then obtain a solution $u_{j,\epsilon}$ of $\overline{\partial} u= v_\epsilon$ on $\mathbb{C}$ satisfying
\begin{eqnarray*}
   \int_{\mathbb{C}} |u_{j,\epsilon}|^2 e^{-\varphi'_j }d\lambda 
&\leq & \frac{6}{(1-\frac{\alpha_j}{8})^2}  \int_{\mathbb{C}} \left |f \overline{\partial}\chi (\cdot)\right|_{\partial \overline{\partial}(\Phi_j+ \Psi_j ) } ^2 e^{ \Psi_j -\Phi_j }d\lambda \nonumber \\
& \le & \frac{C}{\alpha_j}\epsilon^2 \int_{\frac{1}{2\epsilon} \leq  -  \psi \leq \frac{1}{\epsilon}}  |f|^2 e^{-\varphi'_j}d\lambda.\nonumber \\
\end{eqnarray*}
Put $K: = \{z: z\in \mathbb{C}, \ \ -\psi\le  \frac{1}{ \epsilon} \} $. Since $f\in H(\varphi)$, according to the strong openness theorem there exists $j_0$ so that when $j \geq j_0 $ we have $\int_K |f|^2 e^{-\varphi_{j}}d\lambda<\infty.$


Next by using  Lemma \ref{le:guji} and Proposition \ref{prop phipsi} we obtain that for some constant $C,$ independent of $j$ and $K$  
\begin{eqnarray*}
 \int_K |f|^2 e^{-\varphi'_j} d\lambda & =&  \int_K |f|^2 e^{-\frac{\alpha_j}{2}\log(e+|z|^2)  - \widetilde{\varphi}_j} d\lambda  \nonumber \\
& \leq & C \int_K |f|^2 e^{-\int_{|\zeta|<R} \log|z-\zeta| d\mu_j(\zeta) - \widetilde{\varphi}_j} d\lambda \nonumber \\
&= &C \int_K |f|^2 e^{-\varphi_j} d\lambda <\infty.  \nonumber \\
\end{eqnarray*}
Set
$$
F_{j,\epsilon} = f \cdot\chi\left( \log (- \psi)+ \log \epsilon   \right) - u_{j,\epsilon}.
$$
Then $F_{j,\epsilon}$ is an entire function for each $j \geq j_0 \gg1$ with
$$
 \|F_{j,\epsilon}\|_{L^2(\mathbb{C}, \varphi'_j)} \leq (1+ \frac{C}{\sqrt{\alpha_j}} \epsilon ) \| f\|_{L^2(K, \varphi'_j)}  < +\infty.
$$
That is $F_{j,\epsilon }\in \bigcup\limits_{j=1}^\infty H(\varphi'_j) = \bigcup\limits_{j=1}^\infty H(\varphi_j) $. We also obtain
\begin{eqnarray*} \label{eq:estimate12}
 \| F_{j,\epsilon}- f\|^2_{L^2(\mathbb{C}, \varphi)} 
& \leq &  2 \int_{-\psi \geq \frac{1}{2\epsilon}} |f|^2 e^{-\varphi}d \lambda
+  C \int_{\mathbb{C}} | u_{j,\epsilon} |^2 e^{-\varphi'_j} d\lambda   \nonumber \\
& \leq &  2 \int_{-\psi \geq \frac{1}{2\epsilon}} |f|^2 e^{-\varphi}d \lambda+
C'\epsilon^2  \int_K |f|^2 e^{-\varphi_j}d \lambda. \nonumber \\
\end{eqnarray*}
Still keeping $\epsilon$ fixed, but letting $j\rightarrow \infty,$ we get
$$
\limsup_{j\rightarrow \infty} \| F_{j,\epsilon}- f\|^2_{L^2(\mathbb{C}, \varphi)} 
\leq 2 \int_{-\psi \geq \frac{1}{2\epsilon}} |f|^2 e^{-\varphi}d \lambda+
C'\epsilon^2  \int_K |f|^2 e^{-\varphi}d \lambda.
$$
Finally we let $\epsilon \rightarrow 0.$
Then $\bigcup \limits _{j=1}^\infty H(\varphi_j)$ is dense in $H(\varphi)$, which completes the proof.
\end{proof}

For any subharmonic {\bf $\phi$} on $\mathbb{C}$, we let $\mu=\frac{1}{2\pi} \Delta \phi$ which is a locally finite positive measure on $\mathbb C.$ Then $\mu$ decomposes
into a sum $\mu=\mu_1+\mu_2$ where $\mu_2=\sum_i a_i\delta_{z_i}$ is an at most countable
sum of Dirac masses and where $\mu_1$ has no point masses. Thus  Theorem \ref{th:main1} follows as above,
using  Condition (B) because of  the following lemma. 

\begin{lemma}
Suppose $\mu_1$ is not identically zero. Then there exist a point $z_0\in \mathbb C$ and $0<r<s$
so that $\mu(\Delta(z_0,r))>0, \mu(\Delta(z_0,s))< 2.$
\end{lemma}

\begin{proof}
The support of the measure $\mu_1$ is uncountable. Hence we can choose a point $z_0$ in the support of $\mu_1$ which is not one of the $z_i$. Then $\mu(\Delta(z_0,s))\rightarrow 0$ as $s\rightarrow 0$ while
$\mu(\Delta(z_0,r))>0$ for all $r>0.$ 
\end{proof}

\section{Proof of Theorem \ref{th:bergmanconver} and Theorem \ref{Hartogsconvergence}}

\begin{proof}[Proof of theorem \ref{th:bergmanconver}]
We assume that if the upper regulation $\phi$ of $\varphi$ satisfies Condition (B).  For other $\phi$, we may use the same method as in the proof of Theorem \ref{th:main1}, we skip the details. Let $\chi$, $\psi$, $\Psi_j, \Phi_j$ and $K$ as before in the proof of Lemma \ref{ConditionB}. The proof is similar to \cite{boyongChen-20151}. Set 
$$
\lambda_\epsilon = \chi (\log (-\psi) + \log \epsilon),  \ \ \  \epsilon \ll 1.
$$
Let $w \in B_R: = \{ |z| < R \}$. Applying Lemma \ref{Boberndtsson2001} with $\varphi$ and $\psi$ replaced by $\Phi_j$ and $\Psi_j$ respectively, we get a solution $u_{j,\epsilon}$ of

$$
\overline{\partial} u = K_\varphi(\cdot, w) \overline{\partial} \lambda_\epsilon
$$
such that 
$$
\int_{\mathbb{C}} |u_{j,\epsilon}|^2 e^{-\varphi'_j }d\lambda  \leq  \frac{C}{\alpha_j}\epsilon^2  \int_{\frac{1}{2\epsilon} \leq   -  \psi \leq \frac{1}{\epsilon}}  |K_\varphi(\cdot, w)|^2 e^{-\varphi_j} d \lambda \leq  \frac{C'}{\alpha_j}\epsilon^2 K_\varphi(w, w).
$$

The last inequality holds because of the following argument. 

Since $ - \varphi_j$ monotonically decreases to $ - \varphi$, we can find appropriate control function of $ |K_\varphi(\cdot, w)|^2 e^{-\varphi_j}$ on $K$. So we have for all large enough $j$,



\begin{eqnarray*}
\int_K |K_\varphi(\cdot, w)|^2 e^{-\varphi_j} d\lambda \longrightarrow \int_K  |K_\varphi(\cdot, w)|^2 e^{-\varphi} d\lambda 
\end{eqnarray*}
in view of the Lebesgue dominated convergence theorem. 

If $\epsilon \ll 1,$ then $B_{R+1} \subset \{  -\psi \leq \frac{1}{2\epsilon}\}$. Since $u_{j,\epsilon}$ is holomorphic on $\{  -\psi \leq  \frac{ 1}{2\epsilon} \}$, the mean value inequality yields 
\begin{eqnarray*}
 |u_{j,\epsilon}(w)|^2 & \leq &  C_n \int_{B_{R+1}} |u_{j,\epsilon}|^2 d\lambda  \nonumber \\
&\leq & C_{n, R} \int_{B_{R+1}}  |u_{j,\epsilon}|^2e^{ -\varphi_j}d\lambda \nonumber \\
&\leq & C'_{n, R} \int_{B_{R+1}}  |u_{j,\epsilon}|^2e^{ -\varphi'_j}d\lambda \nonumber \\
& \le & \frac{C''_{n, R}}{ \alpha_j}\epsilon^2 K_\varphi(w, w).\nonumber \\
\end{eqnarray*}
It follows that 
$$
f_{j,\epsilon} := \lambda_\epsilon K_\varphi(\cdot, w) - u_{j,\epsilon}
$$
is an entire function satisfying 
$$
|f_{j,\epsilon} (w)| \geq K_\varphi(w, w)  -\frac{C_{n,R}}{\sqrt{\alpha_j}} \epsilon
$$
and

\begin{eqnarray*}
 \|f_{j,\epsilon}\|_{H(\varphi_j)} &\leq & \|K_\varphi(\cdot, w)\|_{L^2(K, \varphi_j)}  +C  \|u_{j,\epsilon}\|_{L^2 (\mathbb{C}, \varphi_j)}  \nonumber \\
& \leq&  (1+ \frac{C}{\sqrt{\alpha_j}} \epsilon ) \|K_\varphi(\cdot, w)\|_{L^2(K, \varphi_j)}\nonumber \\ 
& \leq&  (1+ \frac{C}{\sqrt{\alpha_j}} \epsilon ) \sqrt{ K_\varphi(w, w)}. \nonumber \\ 
\end{eqnarray*}
Thus we have 
$$
\frac{ |f_{j,\epsilon}(w) |}{  \| f_{j,\epsilon}\|_{H(\varphi_j)}} \geq \frac{K_\varphi(w, w) - C_{n,R, \alpha_j }\epsilon}{(1+ C_{\alpha_j} \epsilon ) \sqrt{K_\varphi(w, w)} },
$$
that is
$$
\liminf\limits_{j\rightarrow +\infty} K_{\varphi_j}(w, w) \geq K_\varphi(w, w).
$$
Since $\varphi_j \leq \varphi$ we know that $K_{\varphi_j}(w, w) \leq K_\varphi(w, w)$  for each $j \geq1 $, thus we obtain 
$$
\lim\limits_{j\rightarrow +\infty} K_{\varphi_j}(w, w) = K_\varphi(w, w),    \ \ \ \   \forall w \in \mathbb{C}.
$$
This completes the proof.
\end{proof}



\begin{proof} [Proof of Theorem \ref{Hartogsconvergence}]
 For each compact $F\subset \subset \mathbb{C}$, each fixed $w \in  F $ and $z \in F$, according to the mean value inequality we know that 
\begin{eqnarray}  \label{ineq:converge}
 &&|K_{\varphi_j} (z,w) -K_\varphi (z,w)| ^2 \nonumber \\
& \leq &C'  \| K_{\varphi_j} (\cdot, w) -K_\varphi(\cdot, w)\|^2 _{H(U,0)} \nonumber \\
&\leq & C  \| K_{\varphi_j} (\cdot, w) -K_\varphi(\cdot, w)\| ^2_{H(U, \varphi)} \nonumber \\
&\leq & C  \| K_{\varphi_j} (\cdot, w) -K_\varphi(\cdot, w)\| ^2_{H(\mathbb{C}, \varphi)} \nonumber \\
& = &  C \left( \int _{\mathbb{C}} |K_{\varphi_j}(\cdot, w)|^2 e^{-\varphi} d\lambda +  \int _{\mathbb{C}} |K_\varphi(\cdot, w)|^2 e^{-\varphi} d\lambda -2K_{\varphi_j}(w, w)\right) \nonumber \\
& \leq & C\left ( \int _{\mathbb{C}} |K_{\varphi_j}(\cdot, w)|^2 e^{-\varphi_j} d\lambda + K_\varphi(w, w) - 2K_{\varphi_j}(w, w) \right)\nonumber \\
&=& C (K_\varphi(w, w) - K_{\varphi_j}(w, w))
\end{eqnarray}
where $U$ is some neighborhood of the compact set $F$. By using the result of Theorem \ref{th:bergmanconver} we have $K_{\varphi_j}(z, w)$ pointwise converges to  $K_{\varphi}(z, w)$ in 
$\mathbb{C} \times \mathbb{C}$. Similarly we have 
\begin{eqnarray} \label{Bergmanconvergence}
 \lim \limits _{j \rightarrow \infty} K_{2(k+1)\varphi_j}(z,w) =  K_{2(k+1)\varphi} (z,w)   \ \ \  \forall z ,  w \in \mathbb{C}.
\end{eqnarray}
On the other hand, from Ligocka's formula
$$
K_{\Omega_j}[(z,t), (w,s)] = \sum\limits_{k=0}^{\infty} 2(k+1) K_{2(k+1)\varphi_j}(z,w) \langle t, s\rangle^k, \ \  z ,t , w, s \in \mathbb{C},  
$$
we can easily obtain that 
$$
2(k+1) K_{2(k+1)\varphi_j} (z,w) = \frac{\partial ^{2k}}{ \partial t^k \partial \overline{s}^k} K_{\Omega_j}[(z,t), (w,s)]  \rvert _{t=s=0}.
$$
For each $(z_0, t_0) \in \Omega \subset \mathbb{C}^2$, there exist  $r_1, r_2>0$  so that 
$$
(z_0, t_0 ) \in P: = \Delta (z_0, r_1) \times \Delta(0,r_2) \subset \subset \Omega \subset \Omega_j, \ \  j\geq1. 
$$
Since $\varphi_j$ is increasing to $\varphi$ we know that 
for each $j\geq 1$
\begin{eqnarray} \label{Normal family}
|K_{\Omega_j} [(z,t), (w,s)] | & \leq & K_{\Omega_j} [(z,t), (z,t)]^{\frac12} K_{\Omega_j} [(w,s), (w,s)]^{\frac12}\\ \nonumber
& \leq & K_{\Omega} [(z,t), (z,t)]^{\frac12} K_{\Omega} [(w,s), (w,s)]^{\frac12}.
\end{eqnarray}
Put $M: = \sup\limits_j \sup\limits_{\overline{P} \times \overline{P} } |K_{\Omega_j} [(z,t), (w,s)]|, $ we have $M < +\infty.$
According to the Cauchy integral formula we obtain 
$$
\left \lvert  \frac{\partial ^{2k}}{ \partial t^k \partial \overline{s}^k} K_{\Omega_j}[(z,t), (w,s)] \right  \rvert
_{t=s=0}  \leq C_k \frac{M}{ r_2^{2k}}.
$$
Let $0 < r_1'< r_1$, $0< r_2' < r_2$, then for each $\epsilon >0,$ there exists $k_\epsilon \gg1$, so that 
$$
\sum \limits_{k \geq  k_\epsilon} 2(k+1) \left  \lvert  K_{2(k+1)\varphi_j} (z,w)   (t\overline{s})^k \right \rvert < \epsilon ,  \  \forall z,w \in  \Delta (z_0,r'_1),  \ \  \forall t,s \in  \Delta (0,r'_2), 
$$
and 
$$
\sum \limits_{k \geq  k_\epsilon} 2(k+1)   \lvert  K_{2(k+1)\varphi} (z,w)  (t\overline{s})^k  \rvert < \epsilon, \ \  \forall z,w \in  \Delta (z_0,r'_1),  \  \forall t,s \in  \Delta (0,r'_2). 
$$
Thus we get 
$$
\sum \limits_{k =0} ^{k_\epsilon} 2(k+1) K_{2(k+1)\varphi_j} (z,w) ( t\overline{s}) ^k  \longrightarrow \sum \limits_{k =0} ^{k_\epsilon} 2(k+1)   K_{2(k+1)\varphi } (z,w)  (t\overline{s})^k  
$$
by (\ref{Bergmanconvergence}) we have $K_{\Omega_j}[(z,t), (w,s)] $ pointwise converges to $ K_{\Omega}[(z,t), (w,s)] $ in $\Omega \times \Omega$. Because of (\ref{Normal family}) we know that the functions $K_{\Omega_j}[(z,t), (w,s) ]$ form a normal family in $\Omega \times \Omega$. From the normality and the pointwise convergence just proved, it is immediate that the convergence is uniform on compact subsets of $\Omega \times \Omega$. This  completes the proof.
\end{proof}

\textbf{Acknowledgements}  The  authors were supported in part by the Norwegian Research Council grant number 240569, the first author was also supported by NSFC grant 11601120. The first author also thanks Professor Bo-Yong Chen for his valuable comments.

\end{document}